\documentclass[sn-mathphys]{sn-jnl-preprint}% Math and Physical Sciences Reference Style
\usepackage{amsmath,amssymb,amsthm,amsfonts,mathrsfs}
\usepackage{xcolor}
\usepackage{relsize}
\usepackage{enumerate}
\usepackage{colonequals}

\jyear{2022}%

\theoremstyle{thmstyleone}%
\newtheorem{theorem}{Theorem}[section]
\newtheorem{lemma}[theorem]{Lemma}
\newtheorem{proposition}[theorem]{Proposition}% 
\newtheorem{corollary}[theorem]{Corollary}

\theoremstyle{thmstyletwo}%
\newtheorem{remark}{Remark}%

\newcommand{\N}{\mathbb{N}}

\renewcommand{\R}{\mathbb{R}}
\renewcommand{\C}{\mathbb{C}}
\renewcommand{\P}{\mathbb{P}}
\newcommand{\cP}{\mathcal{P}}
\newcommand{\cM}{\mathcal{M}}
\newcommand{\cS}{\mathcal{S}}
\newcommand{\BEu}[1]{\underset{#1}{\mathlarger{\mathlarger{\mathbb{E}}}^{~}}\,}
\newcommand{\BEul}[1]{\underset{#1}{\mathlarger{\mathlarger{\mathbb{E}}}^{\text{\normalfont\footnotesize log}}}\,}
\newcommand{\ve}{\varepsilon}
\newcommand{\on}{\operatorname}
\newcommand{\li}{\on{Li}}
\renewcommand{\mod}[1]{\,(\on{mod}#1)}
\newcommand{\of}[1]{\left(#1\right)}
\renewcommand{\set}[1]{\left\{#1\right\}}

\begin{document}

\title[The Erd\H{o}s-Kac Theorem and the Prime Number Theorem]{Generalizations of the Erd\H{o}s-Kac Theorem and the Prime Number Theorem}

\author[1]{\fnm{Biao} \sur{Wang}}\email{wangbiao@amss.ac.cn}

\author[2]{\fnm{Zhining} \sur{Wei}}\email{wei.863@osu.edu}

\author[3]{\fnm{Pan} \sur{Yan}}\email{panyan@arizona.edu}

\author[4]{\fnm{Shaoyun} \sur{Yi}}\email{yishaoyun926@xmu.edu.cn}

\affil[1]{\orgdiv{Hua Loo-Keng Center for Mathematical Sciences}, \orgname{Academy of Mathematics and Systems Science, Chinese Academy of Sciences}, \orgaddress{\city{Beijing}, \postcode{100190},  \country{China}}}

\affil[2]{\orgdiv{Department of Mathematics}, \orgname{Ohio State University}, \orgaddress{\city{Columbus},  \state{OH} \postcode{43210}, \country{USA}}}

\affil[3]{\orgdiv{Department of Mathematics}, \orgname{University of Arizona}, \orgaddress{\city{Tucson}, \state{AZ} \postcode{85721},  \country{USA}}}

\affil[4]{\orgdiv{School of Mathematical Sciences}, \orgname{Xiamen University}, \orgaddress{\city{Xiamen}, \state{Fujian} \postcode{361005}, \country{China}}}

\abstract{In this paper, we study the linear independence between the distribution of the number of prime factors of integers and that of the largest prime factors of integers. Respectively, under a restriction on the largest prime factors of integers, we will refine the Erd\H{o}s-Kac Theorem and Loyd's recent result on Bergelson and Richter's  dynamical generalizations of the Prime Number Theorem. At the end, we will show that the analogue of these results holds with respect to the  Erd\H{o}s-Pomerance Theorem as well.}

\keywords{Erd\H{o}s-Kac Theorem, Erd\H{o}s-Pomerance Theorem, Largest Prime Factor, Prime Number Theorem}

\pacs[MSC Classification]{11K36, 37A44}

\maketitle

\section{Introduction and statement of results}

Let $\Omega(n)$ be the number of prime factors of $n$ with multiplicity counted. The distribution of $\Omega(n)$ is an intriguing topic in analytic number theory. The well-known Erd\H{o}s-Kac Theorem \cite{ErdosKac1940}  asserts that  $\Omega(n)$ satisfies the following normal distribution
\begin{equation}\label{eqn_EK}
	\lim_{x \to \infty} \frac1x\sum_{1\le n\le x} F \Big( \frac{\Omega(n) - \log \log x }{\sqrt{\log \log x}} \Big) 
	=
	\frac{1}{\sqrt{2\pi}} \int_{-\infty}^{\infty} F(t) e^{-t^2/2} \, dt
\end{equation}
for any $F\in C_c(\R)$, where $C_c(\R)$ denotes the set of  compactly supported continuous functions on $\R$. This result of Erd\H{o}s and Kac ignites the study of probabilistic number theory and  has been widely generalized in the literature (cf. \cite{Elliott1979, Elliott1980} etc.). One of the generalizations is to consider the weighted variant of \eqref{eqn_EK} for multiplicative functions. In 2019, Elboim and Gorodetsky \cite{ElboimGorodetsky2019} showed that 
\begin{equation}\label{eqn_EK_weighted}
\begin{aligned}
	&\lim_{x \to \infty} \Big(\sum_{1\le n\le x}f(n)\Big)^{-1}\sum_{1\le n\le x}f(n) F\Big( \frac{\Omega(n) - \alpha\log \log x }{\sqrt{\alpha\log \log x}} \Big) \\
	&=
	\frac{1}{\sqrt{2\pi}} \int_{-\infty}^{\infty} F(t) e^{-t^2/2} \, dt
\end{aligned}
\end{equation}
for any $F\in C_c(\R)$, where $f\colon \mathbb{N}\to [0,\infty)$ is a multiplicative function satisfying the following two conditions for some real $d>-1$ and $\alpha>0$:
\begin{align}
	{\rm (I).}\quad 	&\sum_{p\le x}\frac{f(p)\log p}{p^d}=\alpha x+O_A\of{\frac{x}{\log^A x}}\quad \text{ for all } A>0; \label{eqn_EG1}\\
	{\rm (II).}\quad	&\frac{f(p^i)}{p^{di}}=O(r^i)\quad \text{ for some } 1\le r< \sqrt2 \text{ and all } i\ge1. \label{eqn_EG2}
\end{align}
We note  that Tenenbaum \cite{Tenenbaum2017, Tenenbaum2017c} also showed \eqref{eqn_EK_weighted} for a class of multiplicative functions; and recently Khan, Milinovich and Subedi \cite{KhanMilinovichSubedi2021} showed the case $f(n)=d_k(n)$ using Granville-Soundararajan's sieve method \cite{GranvilleSoundararajan2007}. Here, $d_k(n)\colonequals\sum_{a_1\cdots a_k=n}1$ is the $k$-th divisor function with $k\in\N$. 

In this paper, we will consider a refinement of  \eqref{eqn_EK_weighted} with the largest prime factors of integers for multiplicative functions behaving like $d_k(n)$. Let $P^+(n)$  denote the largest prime factor of $n$ for $n\ge2$, and set $P^+(1)=1$. In 1977, Alladi \cite{Alladi1977} showed that $P^+(n)$ is equidistributed in arithmetic progressions when he studied an application of duality between the prime factors of integers. Recently, Kural, McDonald and Sah \cite{KuralMcDonaldSah2020} generalized Alladi's result to the natural density over number fields. More precisely, if $S$, a set of primes, has a natural density $\delta(S)$ (see Sect.~\ref{naturaldensity}), then Kural et al. proved the following equidistribution property of $P^+(n)$: 
\begin{equation}\label{eqn_KMS}
	\lim_{x\to\infty}\frac1x\sum_{\substack{1\le n\le x\\ P^+(n)\in S}} 1= \delta(S).
\end{equation}
From \eqref{eqn_KMS}, we see that the largest prime factors of integers are randomly distributed. 
This motivates us to expect that they are linearly independent with the distribution of sensible multiplicative functions  like $\Omega(n)$ and $d_k(n)$. In particular, for a multiplicative function $f(n)$, if it takes a constant value at primes, then it behaves like $d_k(n)$. An arithmetic function $f\colon \N\to\C$ is said to be \textit{divisor-bounded} if there is a fixed integer $k\in \N$ such that $\vert f\vert \le d_k$. Our first result  is a refinement of \eqref{eqn_EK_weighted} for a class of multiplicative functions containing $d_k(n)$ as weights.

\begin{theorem}\label{mainthm_EK}
	Suppose $\alpha>0$ is a positive real number. Let $f$ be a non-negative divisor-bounded  multiplicative function satisfying $f(p)=\alpha$ for all primes $p$. Let $F \in C_c(\R)$. If $S$ is a set of primes of natural density $\delta(S)$, then we have
	\begin{equation}\label{eqn_mainthm_EK}
			\begin{aligned}
		&\lim_{x \to \infty} \Big(\sum_{1\le n\le x}f(n)\Big)^{-1}\sum_{\substack{1\le n\le x\\ P^+(n)\in S}}f(n) F\Big( \frac{\Omega(n) - \alpha\log \log x }{\sqrt{\alpha\log \log x}} \Big) \\
		&=\delta(S)\cdot\Big(\frac{1}{\sqrt{2\pi}} \int_{-\infty}^{\infty} F(t) e^{-t^2/2} \, dt\Big).
		\end{aligned}
	\end{equation}
\end{theorem}

Clearly, the weight function $f$ in Theorem~\ref{mainthm_EK} satisfies conditions \eqref{eqn_EG1} and \eqref{eqn_EG2}. Taking $S$ to be the set of all primes, we recover the weighted Erd\H{o}s-Kac theorem~\eqref{eqn_EK_weighted}. 
Examples of common weight functions in Theorem~\ref{mainthm_EK} are $d_\alpha(n), \mu(n)^2$, and $\alpha^{\omega(n)}$ for $\alpha>0$, where $d_\alpha(n)$ is the divisor function associated to Dirichlet series $\zeta(s)^\alpha$, $\mu(n)$ is the M\"{o}bius function, and $\omega(n)$ is the number of distinct prime factors of $n$. Here $\zeta(s)=\sum_{n=1}^\infty 1/n^s \,(\on{Re} s>1)$ is the Riemann zeta function. We remark that the proof of Theorem~\ref{mainthm_EK} applies to the $\Omega(n)$ replaced by $\omega(n)$ as well, and the result is the same.

Our next goal is to apply Theorem~\ref{mainthm_EK} to establish a new refinement of dynamical generalizations of the Prime Number Theorem (PNT) discovered by Bergelson and Richter \cite{BergelsonRichter2020}. Let $(X,\mu,T)$ be a uniquely ergodic additive topological dynamical system and $C(X)$ the space of continuous functions on $X$. In 2020, Bergelson and Richter \cite{BergelsonRichter2020} generalized the PNT in the setting of dynamical systems:  for every $x_0\in X$ and $g\in C(X)$ we have
\begin{equation}\label{eqn_BR2020thmA}
	\lim_{N\to\infty}\frac1N\sum_{n=1}^Ng(T^{\Omega(n)}x_0)=\int_Xg\,d\mu.
\end{equation}
In particular, taking $X$ to be the two-point rotation system in \eqref{eqn_BR2020thmA}, we get the following result: 
\begin{equation}\label{eqn_Landau}
	\lim_{N\to\infty}\frac1N\sum_{n=1}^N\lambda(n)=0
\end{equation}
for the Liouville function $\lambda(n)\colonequals(-1)^{\Omega(n)}$, which is equivalent to the Prime Number Theorem (e.g., \cite{Landau1953, Mangoldt1897}).
Also, Eq.~\eqref{eqn_BR2020thmA} unifies the classical results of Pillai \cite{Pillai1940}, Selberg \cite{Selberg1939}, Erd\H{o}s \cite{Erdos1946} and Delange \cite{Delange1958}, see the remarks on Theorem A in \cite{BergelsonRichter2020}. We refer the readers  to Bergelson and Richter's work \cite{BergelsonRichter2020} for more applications of \eqref{eqn_BR2020thmA}. 

Recently, Loyd \cite{Loyd2021} generalized  \eqref{eqn_BR2020thmA} to a disjoint form with the Erd\H{o}s-Kac Theorem:
for any $F \in C_c(\R)$, $g\in C(X)$ and every $x_0\in X$  we have
\begin{equation}\label{eqn_Loyd2021}
	\begin{aligned}
	&\lim_{N \to \infty} \frac1N\sum_{n=1}^N F \Big( \frac{\Omega(n) - \log \log N }{\sqrt{\log \log N}} \Big) g(T^{\Omega(n) }x_0) \\
	&=
	\Big(\frac{1}{\sqrt{2\pi}} \int_{-\infty}^{\infty} F(t) e^{-t^2/2} \, dt\Big)\Big( \int_X g \, d\mu \Big).
	\end{aligned}
\end{equation}
As an application of Theorem~\ref{mainthm_EK}, our second result is the following refinement of \eqref{eqn_Loyd2021}:

\begin{theorem}
	\label{mainthm_EKPNT}
	Let $(X, \mu, T)$ be uniquely ergodic. Let $F \in C_c(\R)$. If $S$ is a set of primes of natural density $\delta(S)$, then we have
	\begin{equation}\label{eqn_mainthm_EKPNT}
			\begin{aligned}
		&\lim_{N \to \infty}\frac1N\sum_{\substack{1\le n\le N\\ P^+(n)\in S}} F\Big( \frac{\Omega(n) - \log \log N }{\sqrt{\log \log N}} \Big) g(T^{\Omega(n) }x_0) \\
		&=
		\delta(S)\cdot\Big(\frac{1}{\sqrt{2\pi}} \int_{-\infty}^{\infty} F(t) e^{-t^2/2} \, dt\Big)\Big( \int_X g \, d\mu \Big)
		\end{aligned}
	\end{equation}
	for all $g \in C(X)$ and $x_0 \in X$. 
\end{theorem}
\begin{remark}
	The first author of this paper proved \eqref{eqn_mainthm_EKPNT} for the special case $F=1$ in {\cite[Theorem~1.3]{Wang2021}}. On the other hand, taking $g=1$, we get a refinement of Erd\H{o}s-Kac Theorem~\eqref{eqn_EK} over $P^+(n)\in S$.
\end{remark}

At the end of this paper, we state a variant of Theorem~\ref{mainthm_EKPNT} on Euler's function. Let $\varphi(n)$ be Euler's totient function. In 1985, Erd\H{o}s and Pomerance \cite{ErdosPomerance1985} showed an Erd\H{o}s-Kac type theorem:
\begin{equation}\label{eqn_EP}
	\lim_{N \to \infty} \frac1N\sum_{1\le n\le N} F\Big( \frac{\Omega(\varphi(n)) - \frac12(\log \log N)^2 }{\frac1{\sqrt3}(\log \log N)^{3/2}} \Big)
	=
	\frac{1}{\sqrt{2\pi}} \int_{-\infty}^{\infty} F(t) e^{-t^2/2} \, dt
\end{equation}
for any $F\in C_c(\R)$. Like the Erd\H{o}s-Kac Theorem, Eq.~\eqref{eqn_EP} has been widely studied in the literature (cf. \cite{BassilyKataiWijsmuller1997} etc.). Making some adjustments in the proof of Theorem~\ref{mainthm_EKPNT}, we have the following analogue with respect to the Erd\H{o}s-Pomerance Theorem. 

\begin{theorem}
	\label{mainthm_EPPNT}
	Let $(X, \mu, T)$ be uniquely ergodic. Let $F \in C_c(\R)$. If $S$ is a set of primes of natural density $\delta(S)$, then we have
	\begin{equation}\label{eqn_mainthm_EPPNT}
			\begin{aligned}
		&\lim_{N \to \infty}\frac1N\sum_{\substack{1\le n\le N\\ P^+(n)\in S}} F\Big( \frac{\Omega(\varphi(n)) - \frac12(\log \log N)^2 }{\frac1{\sqrt3}(\log \log N)^{3/2}} \Big) g(T^{\Omega(n) }x_0) \\
		&=
		\delta(S)\cdot\Big(\frac{1}{\sqrt{2\pi}} \int_{-\infty}^{\infty} F(t) e^{-t^2/2} \, dt\Big)\Big( \int_X g \, d\mu \Big)
		\end{aligned}
	\end{equation}
	for all $g \in C(X)$ and $x_0 \in X$. 
\end{theorem}

\begin{remark} Making similar adjustments, one can also get an analogue of Theorem~\ref{mainthm_EPPNT} for $\Omega(\varphi(n))$ replaced by $\omega_1(n)$, where $\omega_1(n)$ is the number of prime factors of $n$ with multiplicity $1$. The Erd\H{o}s-Kac type theorem for this function can be found in the work of  Elma and Liu \cite{ElmaLiu2021}. We leave the investigation of the analogues of Theorem~\ref{mainthm_EKPNT} for other Erd\H{o}s-Kac type theorems (e.g. \cite{BassilyKataiWijsmuller1997}) to the interested readers. 
\end{remark}

\begin{remark}
	In \cite{Wang2021}, the first author proved the analogues of some cases of \eqref{eqn_mainthm_EKPNT} for additive arithmetic semigroups arising from finite fields. We leave the investigation of the analogues of Theorems~\ref{mainthm_EK}-\ref{mainthm_EPPNT} to readers. 
\end{remark}

In Sect.~\ref{sec_preliminaries}, we collect the ingredients for the proof of Theorem~\ref{mainthm_EK}. Then in Sect.~\ref{sec_pf_mainthm_EK}, we prove Theorem~\ref{mainthm_EK} using the ideas in the work of Kural, McDonald and Sah \cite{KuralMcDonaldSah2020}. In Sect.~\ref{sec_pf_mainthm_EKPNT}, we recite Bergelson and Richter's technical lemmas and use them to prove Theorem~\ref{mainthm_EKPNT}. The proof of Theorem~\ref{mainthm_EPPNT} is similar. In Sect.~\ref{sec_pf_EPPNT}, we will clarify the necessary adjustments to make.

\

\textbf{Notation.} We write $1_P$ for the indicator function of the statement $P$. We write $f(x)=O(g(x))$ or $f(x)\ll g(x)$ if there exists some constant $C>0$ such that $\vert f(x)\vert \le C\vert g(x)\vert $ for all $x$. The implied constant $C$ may depend on some other parameters, but it does not depend on the variable $x$. We write $f(x)=o(g(x))$ if for any $\ve>0$ there exists some constant $N>0$ such that $\vert f(x)\vert \le\ve\vert g(x)\vert $ for all $x\ge N$. The letter $p$ always denotes a prime. The constant $c$ appearing in the $O$-terms is a positive constant that may vary from one line to the next.

\section{Nuts and bolts}
\label{sec_preliminaries}

In this section, we introduce the main ingredients that will be used in Sect.~\ref{sec_pf_mainthm_EK} for the proof of Theorem~\ref{mainthm_EK}.

\subsection{Natural density}
\label{naturaldensity}

Let $\cP$ be the set of all primes.  Let $S\subseteq \cP$ be a subset of primes. Let $$\pi_S(x)\colonequals\#\set{p\in S\colon p\le x},$$
and let $\pi(x)=\pi_{\cP}(x)$ be the prime counting function. We say that $S$ has a \textit{natural density} $\delta(S)$ if the following limit exits:
\[
\delta(S)\colonequals\lim_{x\to\infty} \frac{\pi_S(x)}{\pi(x)}.
\]
For example, if  $S=\set{p\in \cP\colon p\equiv a\mod{q}}$ with $(a,q)=1$, then by the Prime Number Theorem in arithmetic progressions we get that $\delta(S)=1/\varphi(q)$. We refer the readers to \cite{KuralMcDonaldSah2020} for more interesting examples. Suppose that $S$ has a natural density $\delta(S)$. By the Prime Number Theorem, we have $\pi_S(x)\sim \delta(S)\li(x)$, where $\li(x)\colonequals\int_2^x\frac{dt}{\log t}$. Let
$$e_S(x)\colonequals\sup_{y\le x}\vert \pi_S(y)-\delta(S)\li(y)\vert .$$
Then $e_S(x)$ is monotonically increasing and satisfies $e_S(x)=o(x/\log x)$. Now,  we define 
$$v_S(x)\colonequals\sup_{y\ge x}\frac{e_S(y)}{y}.$$
Then $v_S(x)$ is monotonically decreasing and $v_S(x)=o(1/\log x)$. Due to the ideas in \cite[Theorem~3.1]{KuralMcDonaldSah2020}, we have the following property for such $v_S(x)$, which will be used in the last step of the proof of  Theorem~\ref{mainthm_EK}. A discrete version of Lemma~\ref{choice_of_y} can be found in \cite[Lemma~4.5]{DuanWangYi2021}.

\begin{lemma}\label{choice_of_y}
	Suppose $v\colon (1,\infty)\to(0,\infty)$ is a  decreasing function such that  $v(x)\log x=o(1)$. Suppose $h\colon (1,\infty)\to(0,\infty)$ is a positive function such that $\lim_{x\to\infty} h(x)=\infty$, $h(x)=o(\log x)$. Then there exists a positive function $y=y(x)$ satisfying $\lim_{x\to\infty} y(x)=\infty$ such that $u=\log x/\log y\to \infty$, $u\le h(x)$, and $v(y(x))\log x=o(1)$,  as $x\to\infty$.	
\end{lemma}

\begin{proof} 
	Since $\lim_{x\to\infty}v(x)\log x=0$, we have $\lim_{x\to\infty}v(x^{\frac1m})\log x=0$ for any integer $m\ge1$. Then there	exists a minimum positive integer constant $C(m)$ such that
	\begin{equation}\label{limtrickeq}
		v(x^{\frac1m})\log x<\frac1m
	\end{equation}
	for all $x> C(m)$. 
	Then $C(m)$ increases as $m$ increases since $v(x^{\frac1m})$ is increasing with respect to  $m$.  Notice   that $\lim_{x\to\infty}\sup\set{m\in\N\colon C(m)<x}=\infty$, we can set
	\[
	\beta(x)=\min\left(\left\lfloor\sqrt{h(x)}\right\rfloor, \sup\set{m\in\N\colon C(m)<x}\right).
	\]
	Here $\lfloor x\rfloor\colonequals\sup\set{n\in\N\colon n\le x}$ is the integral part of $x$. 
	Then $\beta(x)\le h(x)$ for $x$ sufficiently large  and $\lim_{x\to\infty}\beta(x)=\infty$. Moreover, we have $x>C(\beta(x))$ for all $x>C(1)$. By \eqref{limtrickeq}, we have
	$$v(x^{\frac1{\beta(x)}})\log x<\frac1{\beta(x)}$$
	for all $x>C(1)$. It follows that $\lim_{x\to\infty}v(x^{\frac1{\beta(x)}})\log x=0$. Taking  $y(x)=x^{\frac1{\beta(x)}}$,  which means $u=\beta(x)$,  we have that $\log y(x) =\log x/\beta(x)\ge\log x/h(x)\to\infty$, as $x\to\infty$. This implies that $\lim_{x\to\infty} y(x)=\infty$. Thus, the function $y(x)=x^{\frac1{\beta(x)}}$ satisfies the desired properties. 
\end{proof}

\subsection{Divisor-bounded multiplicative functions}
\label{dbfcn}

In this subsection, we cite Granville and Koukoulopoulos's result  on the asymptotic formula for the partial sum of divisor-bounded multiplicative functions whose prime values are $\alpha>0$ on average. They proved it by using the Landau–Selberg–Delange (LSD) method.

\begin{theorem}[{\cite[Theorem~1]{GranvilleKoukoulopoulos2019}}]\label{gk2019}
	Let $f$ be a multiplicative function satisfying
	\begin{equation}\label{hypothesis}
		\sum_{p\le x}f(p)\log p=\alpha x+O\of{\frac{x}{(\log x)^N}}\qquad(x\ge2)
	\end{equation}
	for some $\alpha\in\C$ and some $N>0$ and such that $\vert f\vert \le d_k$ for some positive real number $k$. Then
	\begin{equation}\label{gk2019eq}
		\sum_{n\le x} f(n)=x \sum_{j=0}^J \tilde{c}_j  \frac{ (\log x)^{\alpha-j-1} } {\Gamma(\alpha-j)}
		+  O\Big(x (\log x)^{k-1-N}(\log\log x)^{1_{N=J+1}}\Big),
	\end{equation}
	where $J$ is the largest integer less than $N$, $F(s)=\sum_{n=1}^\infty{f(n)}{n^{-s}}$, and 
	$$\tilde{c}_j=\frac1{j!}\frac{\mathrm{d}^j}{\mathrm{d} s^j}\bigg\vert_{s=1}\frac{(s-1)^\alpha F(s)}{s}$$
	for $0\le j\le J$. The implied constant in \eqref{gk2019eq} depends at most on $\alpha$, $A$, and the implicit constant in \eqref{hypothesis}.
\end{theorem}

We note that	$\tilde{c}_0=C_\alpha(f)\colonequals\prod_{p}(1-1/p)^\alpha\sum_{\nu\ge0}f(p^\nu)/p^\nu$. By Theorem~\ref{gk2019}, we get that
\begin{equation}\label{mainterm}
	\sum_{n\le x} f(n) \sim C_\alpha(f)x(\log x)^{\alpha-1}
\end{equation}
for the functions in Theorem~\ref{mainthm_EK}, as $x\to\infty$.

\subsection{Friable numbers}
\label{sectsn}

An integer $n$ is called a $y$-friable (or $y$-smooth) number if $P^+(n)\le y$.
Let $$\cS(x,y)\colonequals\set{n\in\N\colon 1\le n\le x, P^+(n)\le y}$$ be the set of $y$-friable numbers up to $x$. For any arithmetic function $f(n)$, we define
\begin{equation}\label{DefnofPsif}
	\Psi_f(x,y)\colonequals\sum_{n\in\cS(x,y)}f(n).
\end{equation}

Let $u\colonequals\log x/\log y$. Let $\varrho_\alpha(u)$ be the function  defined to be the unique continuous solution of the difference-differential equation $u\varrho_\alpha'(u)+(1-\alpha)\varrho_\alpha(u)+\alpha\varrho_\alpha(u-1)=0$ for $u>1$ with the initial condition $\varrho_\alpha(u)=u^{\alpha-1}/\Gamma(\alpha)$ for $0< u\le1$. When $\alpha=1$, $\varrho_\alpha(u)=\varrho_1(u)$ is the classical Dickman function. For a sensible function $f$,  one expects the following asymptotic estimate holds for $\Psi_f(x,y)$:
\begin{equation}\label{psi_eq}
	\Psi_f(x,y)\sim C_\alpha(f) x \varrho_\alpha(u) (\log y)^{\alpha -1} \quad \text{as } y\to\infty,
\end{equation}
where $C_\alpha(f)=\prod_{p}(1-1/p)^\alpha\sum_{\nu\ge0}f(p^\nu)/p^\nu$ as in \eqref{mainterm}. Tenenbaum and Wu \cite{TenenbaumWu2003} proved that Eq.~\eqref{psi_eq} holds for a class $\cM_\alpha$ of multiplicative functions. More precisely, they showed the following estimate.

\begin{theorem}[{\cite[Corollary~2.3]{TenenbaumWu2003}}]\label{twthm}
	For $\varepsilon>0$, let 
	$$L_\varepsilon(y)\colonequals\exp \left(\log(y)^{\frac{3}{5}-\varepsilon}\right),\quad
	H_\varepsilon\colonequals\set{(x, y)\colon x\ge 2, 1\le u \le L_\varepsilon(y)}.
	$$
	Then uniformly for $f\in \mathcal{M}_\alpha$ and $(x, y)\in H_\varepsilon$, we have that
	\begin{equation}\label{twthm_eq}
		\Psi_f(x, y)=C_\alpha(f) x \varrho_\alpha(u) (\log y)^{\alpha -1} \left\{1+O\left( \frac{\log(u+1)}{\log y}+\frac{1}{(\log y )^\alpha}  \right) \right\}.
	\end{equation}
\end{theorem}
The definition of $\cM_\alpha$ in Theorem~\ref{twthm} is a little involved, and we refer the readers to \cite{TenenbaumWu2003} for the precise description. 
Instead we will only use the fact that the class $\cM_\alpha$ contains the divisor-bounded multiplicative functions $f$ as in Theorem~\ref{mainthm_EK}. By \cite[Lemma~1(i)]{Song2002}, we have $\varrho_\alpha(u)\ll \exp(-\tfrac12u\log u)$ for $u\ge1$.  Notice that if $1\le u\le \log x/(\log\log x)^2$, then $(x,y)\in H_\ve$ for some $\ve>0$. Thus, we have the following estimate for the functions $f$ as in Theorem~\ref{mainthm_EK}.

\begin{corollary}\label{sn}
	Let $f$ be as in Theorem~\ref{mainthm_EK}, then we have that
	\begin{equation}\label{sn_eq}
		\Psi_{f}(x,y)\ll x(\log y)^{\alpha-1}\exp(-\tfrac12u\log u)
	\end{equation}
	holds uniformly for $1\le u\le \log x/(\log\log x)^2$.
\end{corollary}

\subsection{Largest prime factors}
In this subsection, we cite one of Ivi\'c and Pomerance's results on the largest prime factors $P^+(n)$ of integers $n$. From the following theorem, one can see that $P^+(n)\vert \vert n$ holds for almost all integers $n$. Here, $P^+(n)\vert \vert n$ means that $P^+(n)^2\nmid n$.

\begin{theorem}[{\cite[Theorem~(1.7)]{IvicPomerance1984}}]\label{lpf} 
	For any real number $r>-1$, we have that
	\begin{equation}
		\sum_{\substack{1\le n\le x\\P^+(n)^2\vert n}}\frac1{P^+(n)^r}=
		x\exp\Big\{-(2r+2)^{\frac12}(\log x\log_2 x)^{\frac12}\Big(1+g_r(x)+O\Big(\big(\frac{\log_3 x}{\log_2 x}\big)^3\Big)\Big)\Big\},
	\end{equation}
	where $\log_kx=\log(\log_{k-1}x)$ is the $k$-fold iterated natural logarithm of $x$, and  
	$$g_r(x)=\frac{\log_3x+\log(1+r)-2-\log2}{2\log_2x}\Big(1+\frac2{\log_2x}\Big)-\frac{\big(\log_3x+\log(1+r)-2\big)^2}{8(\log_2x)^2}.$$
\end{theorem}

Taking $r=0$ in Theorem~\ref{lpf} and then using the Cauchy-Schwarz inequality, we get the following estimate.

\begin{corollary}\label{error_term_sq}
	Let $f$ be an  arithmetic function such that
	\begin{equation}\label{error_term_sq_h}
		\sum_{1\le n\le x}\vert f(n)\vert ^2\ll x(\log x)^B
	\end{equation}
	for some positive constant $B$, then
	\begin{equation}\label{error_term_sq_eq}
		\sum_{\substack{1\le n\le x\\P^+(n)^2\vert n}}f(n)=O\of{
			x\exp\Big\{-c(\log x\log\log x)^{\frac12}\Big\}}
	\end{equation}
	for some positive constant $c=c(B)>0$.
\end{corollary}

\section{Proof of Theorem~\ref{mainthm_EK}}
\label{sec_pf_mainthm_EK}

In this section, we use the ideas in the work of Alladi \cite[Theorem~1]{Alladi1977} and Kural et al.  \cite[Theorem~3.1]{KuralMcDonaldSah2020} to prove Theorem~\ref{mainthm_EK}. Put $\psi(n)\colonequals\frac{\Omega(n) - \alpha\log \log x }{\sqrt{\alpha\log \log x}}$ and $h(n)=f(n)F(\psi(n))$ so that 
\[
\sum_{\substack{1\le n\le x \\ P^+(n)\in S}}h(n)=\sum_{\substack{1\le n\le x \\ P^+(n)\in S}} f(n)F\Big( \frac{\Omega(n) - \alpha\log \log x }{\sqrt{\alpha\log \log x}}\Big) .
\]
By the weighted Erd\H{o}s-Kac Theorem~\eqref{eqn_EK_weighted} and the asymptotic estimate \eqref{mainterm}, to prove \eqref{eqn_mainthm_EK}, it suffices to prove that
\begin{equation}\label{eqn_pf_EK}
	\sum_{\substack{1\le n\le x \\ P^+(n)\in S}}h(n)=\delta(S)\sum_{1\le n\le x }h(n)+o(x(\log x)^{\alpha-1}).
\end{equation}

\begin{proof}[Proof of \eqref{eqn_pf_EK}]
	First, we break up the  sum  into two parts. One is restricted over $P^+(n)\vert \vert n$, and the other is restricted  over $P^+(n)^2\vert n$. Notice that $f(n)$  is  a divisor-bounded function, so is $h(n)$ due to the fact that $F$ is a bounded function. Since \eqref{error_term_sq_h} holds for divisor-bounded functions, by Corollary~\ref{error_term_sq}, we get that
	\begin{equation}\label{pf_eq}
		\sum_{\substack{1\le n\le x \\ P^+(n)\in S}}h(n)=\sum_{\substack{1\le n\le x \\ P^+(n)\in S, P^+(n)\vert \vert n }}h(n)+O\of{
			x\exp\Big\{-c(\log x\log\log x)^{\frac12}\Big\}}.
	\end{equation}
	
	For the first term on the right-hand side of \eqref{pf_eq},  we break up it into two parts:
	\begin{align}\label{pf_s1+s2}
		&\sum_{\substack{1\le n\le x \\ P^+(n)\in S, P^+(n)\vert \vert n }}h(n)\nonumber\\
		&=\sum_{\substack{1\le n\le x \\ P^+(n)\in S, P^+(n)\vert \vert n,  P^+(n)<y}}h(n)+\sum_{\substack{1\le n\le x \\ P^+(n)\in S, P^+(n)\vert \vert n, P^+(n)\ge y}}h(n)\nonumber\\
		&\colonequals S_1+S_2,
	\end{align}
	where $y=x^{1/u}$ is to be determined until the end. That is, $u=\log x/\log y$ as in Sect.~\ref{sectsn}.
	
	Clearly, $\vert S_1\vert \ll \Psi_{f}(x,y)$ by the definition of $\Psi_f$ in Sect.~\ref{sectsn}. By Corollary~\ref{sn}, we have
	\begin{equation}\label{pf_s1}
		S_1\ll x(\log y)^{\alpha-1}\exp(-\tfrac12u\log u)
	\end{equation}
	for $1\le u\le \log x/(\log\log x)^2$.
	
	As regards $S_2$, we write it  as follows:
	\begin{align}\label{pf_s2}
		S_2&=\sum_{\substack{y\le p\le x\\ p\in S}}\sum_{\substack{1\le n\le x \\ P^+(n)=p, p\vert \vert n}}h(n)\nonumber\\
		&=\sum_{\substack{y\le p\le x\\ p\in S}} \sum_{\substack{1\le n\le x/p \\ P^+(n)<p}}h(np)\nonumber\\
		&=\sum_{\substack{y\le p\le x\\ p\in S}}f(p)\sum_{\substack{1\le n\le x/p \\ P^+(n)<p}}f(n)F\left(\psi(n)+\frac1{\sqrt{\alpha\log\log x}}\right).
	\end{align}
	
	Since $F$ is uniformly continuous on $\R$,  the following estimate 
	\begin{equation}\label{eqn_keyeqn1}
		F\left(\psi(n)+\frac1{\sqrt{\alpha\log\log x}}\right)=F(\psi(n))+o(1)
	\end{equation}
	holds uniformly for $n$ as $x\to\infty$. Since $f(p)=\alpha$, it follows that
	\begin{equation}\label{eqn_pf_s2}
		S_2=\alpha \sum_{\substack{y\le p\le x\\ p\in S}} \sum_{\substack{1\le n\le x/p \\ P^+(n)<p}}h(n)+o(1)\cdot \sum_{\substack{y\le p\le x\\ p\in S}} \sum_{\substack{1\le n\le x/p \\ P^+(n)<p}}f(n).
	\end{equation}
	
	Using the definition \eqref{DefnofPsif} of $\Psi_h$ and by Corollary~\ref{error_term_sq}, we can write the double summations in the first term of \eqref{eqn_pf_s2} as
	\begin{equation}\label{eqn_pf_s2_first_term}
		\sum_{\substack{y\le p\le x\\ p\in S}} \sum_{\substack{1\le n\le x/p \\ P^+(n)<p}}h(n)=\sum_{\substack{y\le p\le x\\ p\in S}}\Psi_h\of{\frac{x}p,p}+O\of{
			x\exp\Big\{-c(\log x\log\log x)^{\frac12}\Big\}}.
	\end{equation}
	Notice that the double summations in the second term of \eqref{eqn_pf_s2} is bounded by
	$$\sum_{\substack{y\le p\le x\\ p\in S}} \sum_{\substack{1\le n\le x/p \\ P^+(n)<p}}f(n)\le \sum_{1\le n\le x}f(n)\ll x(\log x)^{\alpha-1}.$$
	Thus, $S_2$ can be written in terms of $\Psi_h(x,y)$ as follows
	\begin{equation}\label{eqn_pf_s2_err}
		S_2=\alpha\sum_{\substack{y\le p\le x\\ p\in S}}\Psi_h\of{\frac{x}p,p}+o(x(\log x)^{\alpha-1}).
	\end{equation}
	
	Now, to separate $\delta(S)$ out of $S_2$, we rewrite the summation in \eqref{eqn_pf_s2_err} as
	\begin{equation}\label{pf_s2-2}
		\sum_{\substack{y\le p\le x\\ p\in S}}\Psi_h\of{\frac{x}p,p}=\delta(S)\int_{y}^x \Psi_h\of{\frac{x}t,t}\frac{dt}{\log t}+S_3,
	\end{equation}
	where
	$$S_3\colonequals\sum_{\substack{y\le p\le x\\ p\in S}}\Psi_h\of{\frac{x}p,p}-\delta(S)\int_{y}^x \Psi_h\of{\frac{x}t,t}\frac{dt}{\log t}.$$
	Here, $S_3$ is an error term. To estimate it, we expand out the $\Psi_h$'s by the definition of $\Psi_h$ and then switch the order of summation and integration:
	\begin{align}\label{pf_s3}
		S_3&=\sum_{\substack{y\le p\le x\\ p\in S}}\sum_{\substack{1\le n\le x/p \\ P^+(n)\le  p}}h(n)-\delta(S)\int_{y}^x\Big(\sum_{\substack{1\le n\le x/t\\P^+(n)\le  t}}h(n) \Big)\frac{dt}{\log t}\nonumber\\
		&=\sum_{\substack{1\le n\le x/y\\ P^+(n)\le  x/n}}h(n)\Bigg(\sum_{\substack{P^+(n)\le  p\le x/n\\ p\ge y, p\in S}}1-\delta(S)\int_{\max\set{P^+(n),y}}^{x/n}\frac{dt}{\log t}\Bigg)\nonumber\\
		&=\sum_{\substack{1\le n\le x/y\\ P^+(n)\le  x/n}}h(n)\Bigg(\pi_S\Big(\frac{x}{n}\Big)-\pi_S(\max\set{P^+(n),y})+O(1)\nonumber\\
		&\hspace{3cm}-\delta(S)\li\of{\frac{x}{n}}+\delta(S)\li(\max\set{P^+(n),y})\Bigg)\nonumber\\
		&\ll\sum_{\substack{1\le n\le x/y\\ P^+(n)\le  x/n}}f(n)\Big(e_S\Big(\frac{x}{n}\Big)+e_S(\max\set{P^+(n),y})+O(1)\Big).
	\end{align}
	
	In the last line of \eqref{pf_s3}, we used two facts: (i) $h(n)\ll f(n)$; (ii) $e_S(\max\set{P^+(n),y})\le e_S(x/n)$. The fact (ii) holds since $P^+(n)\le x/n$, $y\le x/n$, and $e_S$ is increasing. Dropping off the restriction $P^+(n)\le x/n$ under the summation in \eqref{pf_s3},
	\begin{equation}
		S_3\ll \sum_{1\le n\le x/y}f(n)\Big(e_S\Big(\frac{x}{n}\Big)+1\Big).
	\end{equation}
	Since $e_S(x/n)\le (x/n) v_S(x/n)$ and $v_S(x/n)\le v_S(y)$ for $n\le x/y$, we get that
	\begin{equation}\label{eqn_s3}
		S_3	\ll xv_S(y)\sum_{1\le n\le x/y}\frac{f(n)}n +\sum_{1\le n\le x/y}f(n). 
	\end{equation}
	
	By the asymptotic estimate \eqref{mainterm}, we have $\sum_{1\le n\le x}f(n)\ll x(\log x)^{\alpha-1}$. Then by partial summations, we have $\sum_{1\le n\le x}f(n)/n\ll (\log x)^{\alpha}$.  So we have the following two estimates
	\begin{align}
		\sum_{1\le n\le x/y}f(n)&\ll \frac{x}{y}\Big(\log\frac{x}{y}\Big)^{\alpha-1}\ll\frac{x(\log x)^{\alpha-1}}{y},\\
		\sum_{1\le n\le x/y}\frac{f(n)}n&\le \sum_{1\le n\le x}\frac{f(n)}n\ll (\log x)^{\alpha}.
	\end{align}
	It follows by \eqref{eqn_s3} that
	\begin{equation}\label{pf_s3_bound}
		S_3	\ll xv_S(y)(\log x)^{\alpha} + \frac{x(\log x)^{\alpha-1}}{y}. 
	\end{equation}
	
	Combining \eqref{pf_eq}-\eqref{pf_s1}, \eqref{eqn_pf_s2_err}-\eqref{pf_s2-2} and \eqref{pf_s3_bound} together, we get that
	\begin{equation}\label{rs}
		\sum_{\substack{1\le n\le x \\ P^+(n)\in S}}h(n)=\alpha\delta(S)\int_{y}^x \Psi_h\of{\frac{x}t,t}\frac{dt}{\log t}+R_S(x,y),
	\end{equation}
	where 
	$$R_S(x,y)\ll x(\log x)^{\alpha-1}\set{\exp(-\tfrac12u\log u)+ v_S(y)\log x+\frac{1}{y}+o(1)}.$$
	In particular, if we take $S=\cP$, then
	\begin{equation}\label{rp}
		\sum_{1\le n\le x }h(n)=\alpha\int_{y}^x \Psi_h\of{\frac{x}t,t}\frac{dt}{\log t}+R_\cP(x,y).
	\end{equation}
	
	Plugging (\ref{rp}) into (\ref{rs}) finally gives us that
	\begin{equation}\label{key_estimate}
		\sum_{\substack{1\le n\le x \\ P^+(n)\in S}}h(n)=\delta(S)\sum_{1\le n\le x }h(n)+R(x,y)
	\end{equation}
	for $1\le u\le \log x/(\log\log x)^2$, where 
	\begin{equation}\label{key_error_term}
		R(x,y)\ll x(\log x)^{\alpha-1}\set{\exp(-\tfrac12u\log u)+ (v_S(y)+v_\cP(y))\log x+\frac{1}{y}+o(1)}.
	\end{equation}
	
	Now, we use Lemma~\ref{choice_of_y} to estimate $R(x,y)$. Notice that $v_S+v_\cP$ is decreasing, $v_S(x)\log x=o(1)$, and $v_{\cP}(x)\log x=o(1)$. Take $h(x)=\log x/(\log\log x)^2$.
	By Lemma~\ref{choice_of_y}, we can choose a positive function $y=y(x)$ such that $y(x)\to\infty$, $u=\log x/\log y\to\infty$, $u\le \log x/(\log\log x)^2$, and $\lim_{x\to\infty}(v_S(y)+v_\cP(y))\log x=0$, as $x\to\infty$. By \eqref{key_error_term}, it follows that $R(x,y(x))=o(x(\log x)^{\alpha-1})$ for such $y$. Thus, \eqref{eqn_pf_EK} follows by \eqref{key_estimate} immediately. This completes the proof of Theorem~\ref{mainthm_EK}. 
\end{proof}

\section{Proof of Theorem~\ref{mainthm_EKPNT}}\label{sec_pf_mainthm_EKPNT}

Let $[s]\colonequals\N\cap[1,s]$ be the set of natural numbers between $1$ and $s$. Then $[N]=\set{1,\dots,N}$. For a finite non-empty  subset $B\subset \N$, the \textit{Ces\`aro average} and the \textit{logarithmic average}  of an arithmetic function $a\colon B\to\C$ over $B$ are defined respectively by
$$
\BEu{n\in B} a(n) \colonequals \frac{1}{\vert B\vert }\sum_{n\in B} a(n)\qquad\text{and }\qquad
\BEul{n\in B} a(n) \colonequals \frac{\sum_{n\in B} {a(n)}/{n}}{\sum_{n\in B}{1}/{n}}.
$$
In the following three theorems, we cite the techniques on $\mathbb{E}$ and $\mathbb{E}^{\log}$ which Bergelson and Richter \cite{BergelsonRichter2020} used to prove the theorem~\eqref{eqn_BR2020thmA}.

\begin{proposition}[{\cite[Proposition~2.1]{BergelsonRichter2020}}]
	\label{prop_BR_inequality}
	Let $B\subset\N$ be a finite and non-empty set of integers. Then for any bounded arithmetic function $a\colon\N\to\C$ with $\vert a\vert \le1$ we have that
	\begin{equation}
		\label{eqn_BR_inequality}
		\limsup_{N\to\infty}\,\left\vert \BEu{n\in [N]} a(n)\,-\,  \BEul{m\in B}\, \BEu{n\in [\frac{N}{m}]} a(mn) \right\vert  \,\le\,\left( \BEul{m\in B}\BEul{n\in B}\Phi(n,m) \right)^{1/2},
	\end{equation}
	where $\Phi(m,n)\colonequals\gcd(m,n)-1$.
\end{proposition}

\begin{lemma}[{\cite[Lemma~2.2]{BergelsonRichter2020}}]
	\label{lem_BR_keylemma1}
	Let $\P_k\colonequals\set{n\in \N\colon \Omega(n)=k}$, $k\ge1$.  For all $\ve\in(0,1)$ and $\rho\in(1,1+\ve]$, there exist two finite and non-empty sets $B_1,B_2\subset\N$ satisfying the following properties:
	\begin{enumerate}[{\rm (i)}]
		\item\label{itm_a}
		$B_1\subset \P_1$ and $B_2\subset \P_2$;
		\item\label{itm_b}
		$\vert B_1\cap [\rho^j,\rho^{j+1})\vert =\vert B_2\cap [\rho^j,\rho^{j+1})\vert $ for all $j\in\N\cup\{0\}$;
		\item\label{itm_c}
		$\mathbb{E}^{\log}_{m\in B_1}\mathbb{E}^{\log}_{n\in B_1} \Phi(m,n)\le \ve$ and $\mathbb{E}^{\log}_{m\in B_2}\mathbb{E}^{\log}_{n\in B_2} \Phi(m,n)\le \ve$.
	\end{enumerate}
\end{lemma}

\begin{lemma}[{\cite[Lemma~2.3]{BergelsonRichter2020}}]
	\label{lem_BR_keylemma2}
	Fix $\ve\in(0,1)$ and $\rho\in(1,1+\ve]$. 
	Let $B_1$ and $B_2$ be two finite non-empty subsets of $\N$ satisfying property {\rm (ii)} in Lemma~\ref{lem_BR_keylemma1}. Then for any bounded arithmetic function $a\colon\N\to \C$ with $\vert a\vert \le 1$ we have 
	\begin{equation}
		\left\vert \,\BEul{p\in B_{1}} \, \BEu{n\in [\frac{N}{p}]} a(n) ~-~ \BEul{q\in B_{2}} \, \BEu{n\in [\frac{N}{q}]} a(n)\,\right\vert  ~\le~ 5\ve.
	\end{equation}
\end{lemma}

Similar to Sect.~\ref{sec_pf_mainthm_EK}, we put $\psi(n)\colonequals\frac{\Omega(n) - \log \log N }{\sqrt{\log \log N}}$ here without causing any ambiguity of notation. In Theorem~\ref{mainthm_EKPNT}, we have two extra parameters $P^+(n)$ and $F(\psi(n))$ compared with \eqref{eqn_BR2020thmA}. The only new ingredient required in the proof of Theorem~\ref{mainthm_EKPNT} is that on the left hand side of \eqref{eqn_mainthm_EKPNT} the right translation on these two parameters causes a minor perturbation only.  More explicitly, we have the following lemma.

\begin{lemma}\label{lem_keylem_EKPNT}
	Let $F_1,a\colon \N\to\C$ be two bounded functions and let $F_2 \in C_c(\R)$.	Then for any integer $m\in\N$, we have that
	\begin{equation}\label{eqn_keylem_EKPNT}
		\BEu{n\in[N]}F_1(P^+(mn))F_2(\psi(mn))a(n)=\BEu{n\in[N]}F_1(P^+(n))F_2(\psi(n))a(n)+o_{N\to\infty}(1).
	\end{equation}
\end{lemma}
\begin{proof} 
	Eq. \eqref{eqn_keylem_EKPNT} is trivial for $m=1$. Suppose $m\ge2$. We break the average up into two parts as follows:
	\begin{align}
		&\quad\BEu{n\in[N]}F_1(P^+(mn))F_2(\psi(mn))a(n)\nonumber\\
		&=\BEu{n\in[N]}1_{P^+(n)\le m}F_1(P^+(mn))F_2(\psi(mn))a(n)\nonumber\\
		&\qquad+\BEu{n\in[N]}1_{P^+(n)> m}F_1(P^+(mn))F_2(\psi(mn))a(n)\nonumber\\
		&=\BEu{n\in[N]}1_{P^+(n)\le m}F_1(P^+(mn))F_2(\psi(mn))a(n)\nonumber\\
		&\qquad+\BEu{n\in[N]}1_{P^+(n)> m}F_1(P^+(n))F_2(\psi(mn))a(n)\nonumber\\
		&\colonequals S_4+S_5.
		\label{pmaxlem_pf_eq1}
	\end{align}
	For the first average, following the proof of \cite[Lemma~4.5]{Wang2021} we have
	$$S_4=O(N^{-c}),$$
	where $c=1/(2\log m)$.  
	
	As regards $S_5$, observe that $\psi(mn)-\psi(n)=\frac{\Omega(m)}{\sqrt{\log\log N}}$. Since $F_2$ is uniformly continuous, we get that
	\begin{equation}\label{eqn_keyeqn2}
		F_2(\psi(mn))=F_2(\psi(n))+o_{N\to\infty}(1)
	\end{equation}
	holds uniformly for $n\le N$ as $N\to \infty$. It follows that
	\begin{equation}
		S_5=\BEu{n\in[N]}1_{P^+(n)> m}F_1(P^+(n))F_2(\psi(n))a(n)+o_{N\to\infty}(1).
	\end{equation}
	
	By \eqref{pmaxlem_pf_eq1}, we obtain 
	\begin{equation}\label{pmaxlem_pf_eq2}
		\begin{aligned}
			&\BEu{n\in[N]}F_1(P^+(mn))F_2(\psi(mn))a(n)\\
			&=\BEu{n\in[N]}1_{P^+(n)> m}F_1(P^+(n))F_2(\psi(n))a(n)+o_{N\to\infty}(1).
		\end{aligned}
	\end{equation}
	Similar to \eqref{pmaxlem_pf_eq1}, by $S_4=O(N^{-c})$ we have
	\begin{equation}\label{pmaxlem_pf_eq3}
		\BEu{n\in[N]}F_1(P^+(n))F_2(\psi(n))a(n)=\BEu{n\in[N]}1_{P^+(n)> m}F_1(P^+(n))F_2(\psi(n))a(n)+O(N^{-c}).
	\end{equation}
	Hence \eqref{eqn_keylem_EKPNT} follows immediately by \eqref{pmaxlem_pf_eq2} and \eqref{pmaxlem_pf_eq3}.
\end{proof}

By Lemma~\ref{lem_keylem_EKPNT}, replacing $F(P^+(n))$ by $F_1(P^+(n))F_2(\psi(n))$ in \cite[Theorem~4.6]{Wang2021}, we get the following theorem. Then Theorem~\ref{mainthm_EKPNT} follows immediately by Theorem~\ref{mainthm_EK} and Theorem~\ref{thm_keythm_pmax} according to the  argument in \cite[Theorem~A]{BergelsonRichter2020} due to Bergelson and Richter. The details are given as follows for readers' convenience.

\begin{theorem}\label{thm_keythm_pmax}
	Let $F_1\colon \N\to\C$ be a bounded function, and let $F_2 \in C_c(\R)$.	Then for any bounded arithmetic function $a\colon \N\to\C$, we have that
	\begin{equation}\label{eqn_keythm_pmax}
		\BEu{n\in[N]}F_1(P^+(n))F_2(\psi(n))a(\Omega(n)+1)=\BEu{n\in[N]}F_1(P^+(n))F_2(\psi(n))a(\Omega(n))+o_{N\to \infty}(1).
	\end{equation}
\end{theorem}

\begin{proof} We may assume that $\vert F_1\vert , \vert F_2\vert \le1$ and $\vert a\vert \le1$. Let $\ve\in(0,1)$ and $\rho\in (1,1+\ve]$.  Let $B_1$ and $B_2$ be two finite non-empty sets satisfying the properties (i)-(iii) in Lemma~\ref{lem_BR_keylemma1}. 
	For \eqref{eqn_keythm_pmax}, we set
	\begin{align*}
	&S_6\colonequals\BEu{n\in[N]}F_1(P^+(n))F_2(\psi(n))a(\Omega(n)+1),\\
	&S_7\colonequals\BEu{n\in[N]}F_1(P^+(n))F_2(\psi(n))a(\Omega(n)).
	\end{align*}
	And we put
	\begin{align*}
		S_{B_1}&\colonequals\BEul{p\in B_1}\, \BEu{n\in [\frac{N}{p}]}F_1(P^+(pn))F_2(\psi(pn))a(\Omega(pn)+1), \\
		\quad S_{B_2}&\colonequals\BEul{q\in B_2}\, \BEu{n\in [\frac{N}{q}]}F_1(P^+(qn))F_2(\psi(qn))a(\Omega(qn))
	\end{align*}
	for $B_1$ and $B_2$ respectively. 
	Then by Proposition~\ref{prop_BR_inequality} and Lemma~\ref{lem_BR_keylemma1} (iii), we get that
	\begin{equation}\label{pmaxpf_eq1}
		\limsup_{N\to\infty}\vert S_6-S_{B_1}\vert \le \ve^{1/2} \quad\text{and}\quad \limsup_{N\to\infty}\vert S_7-S_{B_2}\vert \le \ve^{1/2}.
	\end{equation}
	By Lemma~\ref{lem_BR_keylemma1} (i),  we have $\Omega(pn)=\Omega(n)+1$ and $\Omega(qn)=\Omega(n)+2$ for $p\in B_1, q\in B_2$. It follows that
	\begin{align*}
		S_{B_1}&=\BEul{p\in B_1}\, \BEu{n\in [\frac{N}{p}]}F_1(P^+(pn))F_2(\psi(pn))a(\Omega(n)+2), \\
		\quad S_{B_2}&=\BEul{q\in B_2}\, \BEu{n\in [\frac{N}{q}]}F_1(P^+(qn))F_2(\psi(qn))a(\Omega(n)+2).
	\end{align*}
	
	Since $B_1$ and $B_2$ are finite, by Lemma~\ref{lem_keylem_EKPNT} we have
	\begin{equation}
		\begin{aligned}
		&\BEu{n\in [\frac{N}{m}]}F_1(P^+(mn))F_2(\psi(mn))a(\Omega(n)+2)\\
		&=\BEu{n\in[\frac{N}{m}]}F_1(P^+(n))F_2(\psi(n))a(\Omega(n)+2)+o(1)
		\end{aligned}
	\end{equation}
	for any $m\in B_1\cup B_2$. It follows that
	\begin{align}
		S_{B_1}&=\BEul{p\in B_1}\, \BEu{n\in [\frac{N}{p}]}F_1(P^+(n))F_2(\psi(n))a(\Omega(n)+2)+o(1), \\
		\quad S_{B_2}&=\BEul{q\in B_2}\, \BEu{n\in [\frac{N}{q}]}F_1(P^+(n))F_2(\psi(n))a(\Omega(n)+2)+o(1).
	\end{align}
	Then by Lemma~\ref{lem_BR_keylemma2} and Lemma~\ref{lem_BR_keylemma1} (ii), we get that
	\begin{equation}\label{pmaxpf_eq2}
		\limsup_{N\to\infty} \vert S_{B_1}-S_{B_2}\vert \le5\ve.
	\end{equation}
	
	Combining \eqref{pmaxpf_eq1} and \eqref{pmaxpf_eq2}, we obtain
	\begin{equation}
		\limsup_{N\to\infty} \vert S_6-S_7\vert \le 2\ve^{1/2}+5\ve.
	\end{equation} 
	Since $\ve$ is arbitrarily small,  we conclude that $\limsup_{N\to\infty} \vert S_6-S_7\vert =0$. This completes the proof.
\end{proof}

\begin{proof}[Proof of Theorem~\ref{mainthm_EKPNT}]
	For any $x\in X$ and $F \in C_c(\R)$,
	we define the measure $\mu_N$ on $X$ by
	\[
	\mu_N\colonequals\frac1N\sum_{n=1}^N 1_{P^+(n)\in S}F(\psi(n))\delta_{T^{\Omega(n)}x}
	\]
	for $N \in \N$, where $\delta_y$ denotes the point mass at $y$ for any $y\in X$. Define $$\mu'\colonequals\delta(S)\cdot\Big( \int_{-\infty}^{\infty}F(t)e^{-t^2/2}\,dt\Big) \cdot \mu.$$ 
	Then Eq. \eqref{mainthm_EKPNT} is equivalent to the assertion that  $\mu_N\to\mu'$ in the weak-$\ast$ topology. Since $\mu$ is uniquely ergodic, it suffices to show  the following $T$-invariance
	\begin{equation}\label{eqn_T-invariance}
		\lim_{N \to \infty} \Big\vert  \int_X g \circ T\, d\mu_N - \int_X g  \, d\mu_N  \Big\vert  = 0
	\end{equation}
	for all $g\in C(X)$. Actually, if we take $F_1(n)=1_{n\in S}$ and $a(n)=g(T^nx_0)$ in Theorem~\ref{thm_keythm_pmax}, then we get the following equivalent form of \eqref{eqn_T-invariance}:
	\begin{multline}
	    \lim_{N \to \infty} \Big\vert  \frac1N\sum_{n=1}^N1_{P^+(n)\in S}F(\psi(n))g(T^{\Omega(n)+1}x_0)\\
	    - \frac1N\sum_{n=1}^N1_{P^+(n)\in S}F(\psi(n))g(T^{\Omega(n)}x_0) \Big\vert  = 0.
	\end{multline}
	Thus, this completes the proof of Theorem~\ref{mainthm_EKPNT}.
\end{proof}

\section{Proof of Theorem~\ref{mainthm_EPPNT}}
\label{sec_pf_EPPNT}

The proof of Theorem~\ref{mainthm_EPPNT} is similar to that of Theorem~\ref{mainthm_EKPNT}. First, we show that
\begin{equation}\label{eqn_EP_pmax}
	\lim_{x \to \infty}\frac1x\sum_{\substack{1\le n\le x\\ P^+(n)\in S}} F\Big( \frac{\Omega(\varphi(n)) - \frac12(\log \log x)^2 }{\frac1{\sqrt3}(\log \log x)^{3/2}} \Big) 
	=
	\delta(S)\cdot\Big(\frac{1}{\sqrt{2\pi}} \int_{-\infty}^{\infty} F(t) e^{-t^2/2} \, dt\Big)
\end{equation}
for $F\in C_c(\R)$. To prove \eqref{eqn_EP_pmax}, we make some adjustments in the proof of Theorem~\ref{mainthm_EK} in Sect.~\ref{sec_pf_mainthm_EK}. Similar to Sect.~\ref{sec_pf_mainthm_EK}, we put $\psi(n)\colonequals \frac{\Omega(\varphi(n)) - \frac12(\log \log x)^2 }{\frac1{\sqrt3}(\log \log x)^{3/2}}$ here,  then it suffices to show that
\begin{equation}\label{eqn_pf_EP_goal}
	\sum_{\substack{1\le n\le x\\ P^+(n)\in S}} F(\psi(n))=\delta(S) \sum_{1\le n\le x}F(\psi(n))+o(x).
\end{equation}

Similar to \eqref{pf_eq}-\eqref{pf_s2}, we have
\begin{equation}
	\sum_{\substack{1\le n\le x\\ P^+(n)\in S}} F(\psi(n))= S_8+O\set{x\exp(-\tfrac12u\log u)+x\exp\big(-c(\log x\log\log x)^{\frac12}\big)}
\end{equation}
for $1\le u\le \log x/(\log\log x)^2$, where
$$S_8\colonequals\sum_{\substack{y\le p\le x\\ p\in S}}\sum_{\substack{1\le n\le x/p \\ P^+(n)<p}}F\left(\psi(n)+\frac{\Omega(p-1)}{\frac1{\sqrt3}(\log \log x)^{3/2}}\right).$$

Let $\ve$ be an arbitrarily small positive. Since $F$ is in $C_c(\R)$, it is uniformly coninuous on $\R$.  There exists a positive $\eta=\eta(\ve)>0$ such that 
$$\vert F(x)-F(y)\vert <\ve$$
for all $x,y\in\R$ with $\vert x-y\vert \le  \eta$. Also, $F$ is bounded, so we may assume that $\vert F\vert \le C$ for some constant $C>0$. Now, we set $\eta_0\colonequals(\log\log x)^{-1/4}$. Suppose $x$ is large enough such that $\eta_0<\eta$. Then we break the following difference on $S_8$ up into two parts:
\begin{align}
	&\quad\Big\vert S_8-\sum_{\substack{y\le p\le x\\ p\in S}}\sum_{\substack{1\le n\le x/p \\ P^+(n)<p}}F(\psi(n))\Big\vert\nonumber\\
	&\le\sum_{\substack{y\le p\le x\\ p\in S}}\sum_{\substack{1\le n\le x/p \\ P^+(n)<p}}\Big\vert F\left(\psi(n)+\frac{\Omega(p-1)}{\frac1{\sqrt3}(\log \log x)^{3/2}}\right)-F(\psi(n))\Big\vert\nonumber\\
	&\le \sum_{\substack{y\le p\le x\\ \frac{\Omega(p-1)}{\frac1{\sqrt3}(\log \log x)^{3/2}}\le \eta_0}}\sum_{\substack{1\le n\le x/p \\ P^+(n)<p}} \ve +\sum_{\substack{y\le p\le x\\ \frac{\Omega(p-1)}{\frac1{\sqrt3}(\log \log x)^{3/2}}>\eta_0}}\sum_{\substack{1\le n\le x/p \\ P^+(n)<p}} 2C \nonumber\\
	&\le \ve\sum_{p\le x}\sum_{\substack{1\le n\le x/p \\ P^+(n)<p}} 1 + 2C\sum_{\substack{y\le p\le x\\ \frac{\Omega(p-1)}{\frac1{\sqrt3}(\log \log x)^{3/2}}>\eta_0}}\frac{x}p \nonumber\\
	&\le \ve x+ \frac{2\sqrt{3}Cx}{\eta_0(\log \log x)^{3/2}} \sum_{y\le p\le x} \frac{\Omega(p-1)}p. \label{eqn_pf_EP_error_term}
\end{align}

By \cite[Lemma~2.3]{ErdosPomerance1985}, we have
\begin{equation}\label{eqn_pf_EP_summation}
	\sum_{1\le p\le x} \frac{\Omega(p-1)}p=\frac12(\log\log x)^2+O(\log\log x).
\end{equation}
Recalling $u=\log x/\log y$, it follows by \eqref{eqn_pf_EP_summation} that
\begin{equation}
	\sum_{y\le p\le x} \frac{\Omega(p-1)}p=\log u\log\log x-\frac12(\log u)^2+O(\log\log x).
\end{equation}
This is bounded by $O\of{\log\log x\log \log \log x}$, provided that $1\le u\le \log \log x$. 

Suppose $1\le u\le \log \log x$. By \eqref{eqn_pf_EP_error_term} and $\eta_0=(\log\log x)^{-1/4}$, 
\begin{align}
	\Big\vert S_8-\sum_{\substack{y\le p\le x\\ p\in S}}\sum_{\substack{1\le n\le x/p \\ P^+(n)<p}}F(\psi(n))\Big\vert\le \ve x+O\of{\frac{x \log \log \log x}{(\log\log x)^{1/4}}}.
\end{align}
Since $\ve$ is arbitrarily small, we obtain that
\begin{equation}
	S_8=\sum_{\substack{y\le p\le x\\ p\in S}}\sum_{\substack{1\le n\le x/p \\ P^+(n)<p}}F(\psi(n))+o(x).
\end{equation}
Then applying the argument from \eqref{eqn_pf_s2_first_term}-\eqref{key_estimate}, we get that
\begin{equation}
	\sum_{\substack{1\le n\le x\\ P^+(n)\in S}} F(\psi(n))=\delta(S) \sum_{1\le n\le x}F(\psi(n))+R(x,y)
\end{equation}
with
$$R(x,y)\ll  x \of{\exp(-\tfrac12u\log u)+ (v_S(y)+v_\cP(y))\log x+\frac{1}{y}+o(1)}.$$
Applying Lemma~\ref{choice_of_y} again, there exists $y=y(x)$ satisfying $u\le \log\log x$ such that $R(x,y)=o(x)$. Hence \eqref{eqn_pf_EP_goal} follows, and \eqref{eqn_EP_pmax} holds.

Now, for Euler's function $\varphi(n)$, we have the following multiplicative identity
$$\varphi(mn)=\varphi(m)\varphi(n)\frac{\gcd(m,n)}{\varphi(\gcd(m,n))}$$
for any $m,n\in \N$. From this identity we can see that 
$$\Omega(\varphi(mn))=\Omega(\varphi(n))+A$$
for some constant $A$ depending on $m$ and divisors of $m$ only. It follows that
\eqref{eqn_keyeqn2} holds with respect to $\Omega(\varphi(n))$, and Lemma~\ref{lem_keylem_EKPNT} and Theorem~\ref{thm_keythm_pmax} hold with respect to $\Omega(\varphi(n))$ as well. Hence Theorem~\ref{mainthm_EPPNT} follows similarly by the argument of Theorem~\ref{mainthm_EKPNT}. \qed

\bmhead{Acknowledgements}
This work is supported by the National Natural Science Foundation of China (No. 12288201). The first author  is supported by the China Postdoctoral Science Foundation under grant number 2021TQ0350. The authors would like to thank Liyang Yang for his helpful discussions and comments. We are also grateful to the anonymous referees for the helpful corrections and suggestions. 

\bmhead{Conflict of interest statement}
The authors  certify that they have no affiliations with or involvement in any organization or entity with any financial interest or nonfinancial interest in the subject matter or materials discussed in this manuscript.


\end{document}